\pdfoutput=1
\documentclass[11pt,reqno,twoside,final]{amsart}

\usepackage[T1]{fontenc}
\usepackage[utf8]{inputenc}
\usepackage{stmaryrd}

\usepackage{geometry}
\usepackage{setspace}
\usepackage{indentfirst}

\usepackage{color}
\usepackage[color,notref,notcite]{showkeys}

\usepackage{amssymb,latexsym,amsmath,amsthm,mathrsfs}
\usepackage{tikz-cd}

\theoremstyle{plain}
\newtheorem{theorem}{Theorem}[section]
\newtheorem{lemma}[theorem]{Lemma}
\newtheorem{definition}[theorem]{Definition}

\newtheorem{proposition}[theorem]{Proposition}
\newtheorem{corollary}[theorem]{Corollary}

\newtheorem{remark}[theorem]{Remark}

\DeclareMathOperator{\pbw}{pbw}
\DeclareMathOperator{\poly}{poly}
\DeclareMathOperator{\End}{End}
\DeclareMathOperator{\id}{id}
\DeclareMathOperator{\Gr}{Gr}
\DeclareMathOperator{\sym}{sym}
\DeclareMathOperator{\un}{\mathscr{A}}
\DeclareMathOperator{\deux}{\mathscr{B}}
\DeclareMathOperator{\trois}{\mathscr{D}}
\DeclareMathOperator{\quatre}{\mathscr{C}}

\newcommand{\M}{{\mathcal{M}}}
\newcommand{\EE}{{\mathcal{E}}}
\newcommand{\px}{\partial_x}

\newcommand{\lie}[2]{\left[#1,#2\right]}
\newcommand{\contraction}[2]{\left\langle#1\middle| #2\right\rangle}
\newcommand{\dual}{^{\vee}}
\newcommand{\SM}{\hat{S}(T\dual_{\M})}
\newcommand{\Sm}{\hat{S}(T\dual_{M})}
\newcommand{\inv}{^{-1}}
\newcommand{\cntn}{\nabla^\lightning}
\newcommand{\coder}{d^\nabla}
\newcommand{\DD}{D}
\newcommand{\nP}{{\Theta^\nabla}}
\newcommand{\nA}{{\Xi^\nabla}}
\newcommand{\Anabla}{{A^\nabla}}
\newcommand{\ds}{d^{\cntn}}
\newcommand{\abs}[1]{\left|#1\right|}
\newcommand{\degree}[1]{\abs{#1}}
\newcommand{\truncation}[1]{\tau_{#1}}
\newcommand{\sections}[1]{\Gamma\left(#1\right)}
\newcommand{\enveloping}[1]{\mathcal{U}(#1)}
\newcommand{\ZZ}{\mathbb{Z}}
\newcommand{\NN}{\mathbb{N}}
\newcommand{\NO}{\NN_0}
\newcommand{\xto}{\xrightarrow}
\newcommand{\perturbation}{\partial}

\begin{document}

\title{Formal exponential map for graded manifolds}
\author{Hsuan-Yi Liao} 
\address{Department of Mathematics, Pennsylvania State University}
\email{hul170@psu.edu}

\author{Mathieu Stiénon}
\address{Department of Mathematics, Pennsylvania State University}
\email{stienon@psu.edu}

\begin{abstract}
We introduce, for every $\mathbb{Z}$-graded manifold, a formal exponential map defined in a purely algebraic way and study its properties. As an application, we give a simple new 
construction of a Fedosov type resolution of the algebra of smooth functions of $\mathbb{Z}$-graded manifolds and we extend the Emmrich--Weinstein theorem to the context of $\mathbb{Z}$-graded manifolds.
\end{abstract}

\thanks{This work was partially supported by  the National Science Foundation [DMS1101827].}

\maketitle

\section*{Introduction}

We introduce, for every $\ZZ$-graded manifold, a formal exponential map defined in a purely algebraic way and advocate for its use in applications.

Although the geodesic exponential map $\exp:T_M\to M\times M$ 
associated to an affine  connection $\nabla$ does not transpose 
in a straightforward way to the graded manifold context, 
its fiberwise infinite-order jet evaluated along the zero section of $T_M$ 
admits a purely algebraic description which does carry over to the $\ZZ$-graded context. 
We prove that the resulting map is an algebraically well-defined isomorphism 
of \emph{co}algebras $\pbw:\Gamma\big(S(T_\M)\big)\to\mathcal{D}(\M)$,
where $\mathcal{D}(\M)$ denotes the associative algebra of differential operators 
on $\M$, called ``\emph{formal exponential map}.'' 
This map may be considered as a replacement for the geodesic exponential map 
of classical differential geometry. 
For instance, the formal exponential map  was exploited recently in relation with 
the Atiyah class of dg-manifolds, which is the obstruction class of the existence 
of dg-compatible connections \cite{MR3319134}. 
Moreover, it plays an important role in
the construction of associated $L_\infty$-algebras~\cite{MR3319134,1408.2903}.
We refer the reader to Definition~\ref{Tuileries} for the precise definition 
of the formal exponential map. 

Applying the iterative technique conceived by Fedosov~\cite{MR1293654} 
for the deformation quantization~\cite{MR0496157,MR0496158} of symplectic manifolds, 
Emmrich \& Weinstein~\cite{MR1327535} 
constructed, 
for every smooth manifold $M$, a flat connection 
on the completed symmetric tensor algebra $\Sm$ of the cotangent bundle of $M$.
Furthermore, they proved, once again by the Fedosov iterative method, that
for any smooth function $f$ on $M$, there exists a unique 
flat section of $\Sm$ whose term of degree $0$ (for the natural
graduation of $\Sm$ determined by the symmetric  tensor power)
is equal to $f$. Doing so, they obtained an augmentation map 
$\tau: C^\infty(M) \to \Omega^0(M,\hat{S}(T\dual_M))$.
Moreover, Emmrich \& Weinstein~\cite{MR1327535}
proved that this map $\tau$ constructed by the Fedosov iterative method
does coincide with the fiberwise infinite-order jet 
along the zero section of $T_M$ of the classical geodesic exponential map
associated to an affine connection on the manifold $M$. Their proof
resorted to a complicated argument involving Ehresmann connections
on analytic manifolds. 
One application of our formal exponential map is 
a direct and much more transparent proof of the Emmrich--Weinstein theorem.
Indeed, we give a simple proof of an extension of the Emmrich--Weinstein theorem
to $\ZZ$-graded manifolds, of which the classical Emmrich--Weinstein theorem 
is a special case.  

First, we extend the Dolgushev construction of flat connections by Fedosov's 
iterative method to the realm of $\ZZ$-graded manifolds. 
Next, we show that, for a graded manifold $\M$, the flat connection $D$ on $\SM$ 
constructed by this iteration procedure \`a la Fedosov starting from a torsionfree 
connection $\nabla$ on $T_\M$ can be recovered in a straightforward manner 
by making use of the `formal exponential map' 
$\pbw:\Gamma\big(S(T_\M)\big)\to\mathcal{D}(\M)$ 
associated with the chosen torsionfree connection $\nabla$ on $T_\M$. 
Our construction goes as follows. 
The Lie algebra $\Gamma(T_\M)$ of smooth vector fields on $\M$ 
acts infinitesimally from the left on $\mathcal{D}(\M)$ by composition 
of differential operators. Identifying $\mathcal{D}(\M)$ with 
the symmetric algebra $\Gamma\big(S(T_\M)\big)$ via the 
formal exponential map $\pbw$ and transferring this infinitesimal action 
through $\pbw$, we obtain a flat connection $\cntn$ on the vector bundle $S(T_\M)$. 
We prove that the covariant differential of the flat connection induced by $\cntn$ on 
the dual bundle $\SM$ coincides with the coboundary operator $D$ constructed 
by iteration --- see Theorem~\ref{Bastille}.  
As a consequence, we prove an extension of the Emmrich--Weinstein theorem 
\cite{MR1327535} to $\ZZ$-graded manifolds: 
the augmentation map $\tau$ which identifies smooth functions on $\M$ 
with $\cntn$-flat sections of the bundle $\SM$ is a Taylor expansion 
twisted by the formal exponential map --- see Corollary~\ref{Bourse}.
When $\M$ is an ordinary smooth manifold, we recover the classical
Emmrich--Weinstein theorem \cite{MR1327535}.

In 2005, while globalizing Kontsevich's formality theorem
~\cite{MR2062626} from local charts to whole smooth manifolds,
Dolgushev~\cite{MR2102846} proved that the Fedosov flat connection $D$ on $\Sm$ 
and the augmentation map $\tau: C^\infty(M) \to \Omega^0(M,\hat{S}(T\dual_M))$, 
which both stem from a choice of torsionfree connection $\nabla$ on the tangent bundle 
$T_M$ (see Propositions~\ref{Concorde} and~\ref{Voltaire}), 
fit into an exact sequence  
\begin{equation}\label{Cadet} 
\begin{tikzcd}[column sep=scriptsize]
0 \arrow{r} & C^\infty(M) \arrow{r}{\tau} 
& \Omega^0(M,\hat{S}(T\dual_M)) \arrow{r}{\DD} 
& \Omega^1(M,\hat{S}(T\dual_M)) \arrow{r}{\DD} 
& \Omega^2(M,\hat{S}(T\dual_M)) \arrow{r}{\DD} 
& \cdots 
\end{tikzcd} \end{equation}
providing a resolution of the algebra $C^\infty(M)$ of smooth functions on $M$. 

An analogue of Dolgushev's result in the context of $\ZZ$-graded manifolds can be 
found in~\cite[Appendix]{MR2304327}, where it is used in connection with the 
quantization of coisotropic submanifolds. 
However, rather than resolve the entire algebra of functions on the graded 
manifold at hand, Cattaneo and Felder only resolve the functions on the support 
of this graded manifold as that is sufficient for their purpose.

In this paper, we give a simple proof of the exactness of the 
sequence~\eqref{Cadet} generalized to $\ZZ$-graded manifolds based on 
homological perturbation. 
We briefly recall the principle of homological perturbation in the Appendix. 
The resulting Dolgushev--Fedosov resolution of $C^\infty(\M)$ 
can be used to globalize Kontsevich's local formality theorem 
(see~\cite{MR2062626}) in the context of $\ZZ$-graded manifolds 
(see~\cite{MR2304327}). 

Finally, we note that since $\Omega^\bullet(\M,\SM)$ may be regarded 
as the algebra of functions on the graded manifold 
$T_\M[1]\oplus T_\M$, the (graded manifold version of) exact sequence 
\eqref{Cadet} means that the dg-manifold $\M$ 
with support $M$ and trivial homological vector field is weakly equivalent to the dg-manifold 
$T_\M[1]\oplus T_\M$ with support $M$ and the operator $D$ as homological vector field. 
Throughout this paper, the structure sheaf $\mathcal{O}_{T_\M}$ of the $\ZZ$-graded manifold $T_\M$ 
is understood to be a sheaf over the smooth manifold $M$. 
Likewise, the support of the graded manifold $T_\M[1]\oplus T_\M$ is 
the smooth manifold $M$.

\subsection*{Notations}

Some remarks concerning notations are necessary. 

By default, in this paper, ``graded'' means $\ZZ$-graded.

We use the symbol $\Bbbk$ to denote the field of either real or complex numbers. 

Given a module $M$ over a ring, 
the symbol $\hat{S}(M)$ denotes the 
$\mathfrak{m}$-adic completion of the symmetric algebra $S(M)$, where $\mathfrak{m}$ is the ideal of $S(M)$ 
generated by $M$. 

The \emph{Koszul sign} $\epsilon(\sigma;X_1,X_2,\cdots,X_n)$ 
of a permutation $\sigma$ of homogeneous elements 
$X_1,X_2,\dots,X_n$ of a graded vector space $V$ 
is determined by the equality
\[ X_{\sigma(1)} \odot X_{\sigma(2)} \odot \cdots \odot 
X_{\sigma(n)} = \epsilon(\sigma;X_1,X_2,\cdots,X_n)\ 
X_1 \odot X_2 \odot \cdots \odot 
X_n \] in the (graded) symmetric algebra $S(V)$. 

Let $\M$ be a finite-dimensional graded manifold,
let $(x_i)_{i\in\{1,\dots,n\}}$ be a set of local coordinates on $\M$ 
and let $(y_j)_{j\in\{1,\dots,n\}}$ be the induced local frame of $T_\M\dual$ 
regarded as fiberwise linear functions on $T_\M$.

We use the symbol $\NN$ to denote the set of positive integers and the symbol $\NO$ for the set of nonnegative integers. 
Given a multi-index $I=(i_1,i_2,\cdots,i_n)\in\NO^n$, 
we adopt the following multi-index notations:
\begin{align*}
I! &= i_1!i_2! \cdots i_n! & 
\truncation{\leqslant k}I &= (i_1,\cdots,i_k,0,\cdots,0) \\
\abs{I} &= i_1+i_2+\cdots+i_n & 
\truncation{< k}I &= (i_1,\cdots,i_{k-1},0,\cdots,0) \\ 
y^I &= (y_1)^{i_1}(y_2)^{i_2} \cdots (y_n)^{i_n} & 
\truncation{>k}I &= (0,\cdots,0,i_{k+1},\cdots,i_n)  \\ 
\end{align*}
\begin{gather*}
\px^I=\underset{i_1 \text{ factors}}{\underbrace{\partial_{x_1}\odot\cdots\odot\partial_{x_1}}}
\odot\underset{i_2 \text{ factors}}{\underbrace{\partial_{x_2}\odot\cdots\odot\partial_{x_2}}}
\odot\cdots\odot\underset{i_n \text{ factors}}{\underbrace{\partial_{x_n}\odot\cdots\odot\partial_{x_n}}}
\\ 
\underleftarrow{\px^I}=\underset{i_n \text{ factors}}{\underbrace{\partial_{x_n}\odot\cdots\odot\partial_{x_n}}}
\odot\underset{i_{n-1} \text{ factors}}{\underbrace{\partial_{x_{n-1}}\odot\cdots\odot\partial_{x_{n-1}}}}
\odot\cdots\odot\underset{i_1 \text{ factors}}{\underbrace{\partial_{x_1}\odot\cdots\odot\partial_{x_1}}}
\end{gather*} 
We use the symbol $e_k$ to denote the multi-index all of whose 
components are equal to $0$ except for the $k$-th which is equal to $1$.  
Thus $\partial_x^{e_k}=\partial_{x_k}$.

The de~Rham exterior differential $d$ is an operator of degree $+1$ while the interior product $i_X$ with a homogeneous 
vector field $X$ of degree $\degree{X}$ is an operator of degree $-1-\degree{X}$. 
The element 
\[ dx_{i_1}\wedge\cdots\wedge dx_{i_p}\otimes y^J\tfrac{\partial}{\partial y_q} \] 
of $\Omega^p\big(\M,S^{\degree{J}}(T\dual_{\M})\otimes T_{\M}\big)$ 
is of degree 
\[ \sum_{k=1}^p(1+\degree{x_{i_k}})+\sum_{k=1}^n J_k\degree{y_k}-\degree{y_q} ,\]
where $\degree{x_{k}}$ (resp.\ $\degree{y_{q}}$) denotes the degree 
of the coordinate function $x_{k}$ (resp.\ $y_{q}$). 

\subsection*{Acknowledgments} 
We would like to express our gratitude to Ping Xu for enlightening discussions. 

\section{Connections on graded manifolds}

A $\ZZ$-graded manifold $\M$ consists of a smooth manifold $M$ (called the support of the graded manifold) and a sheaf of $\ZZ$-graded, graded-commutative algebras $\mathcal{O}_{\M}$ over $M$ such that every point of $M$ admits an open neighborhood $U$ for which $\mathcal{O}_{\M}(U)$ is isomorphic to 
$C^\infty(U)\otimes\hat{S}(V\dual)$, where $V$ is a fixed $\ZZ$-graded vector space and $\hat{S}(V\dual)$ denotes the formal power series on $V$. We say that the graded manifold $\M$ is finite-dimensional if $\dim M$ and $\dim V$ are finite. We use the notation $C^\infty(\M)$ to denote the algebra $\mathcal{O}_{\M}(M)$. We refer the reader to~\cite[Chapter 2]{math/0605356} for a short introduction to $\ZZ$-graded manifolds. 

\begin{definition}\label{Philadelphia}
Let $\EE \to \M$ be a vector bundle in the category of graded manifolds. 
A connection on $\EE \to \M$ is a $\Bbbk$-linear map 
\[ \nabla: \Gamma(T_\M) \otimes \Gamma(\EE) \to \Gamma(\EE) \] of degree $0$ such that 
\begin{gather*}
\nabla_{fX}S = f\nabla_X S ,\\ 
\nabla_X (fS) = X(f)S+(-1)^{|X||f|} f\nabla_X S ,
\end{gather*}
for all $f \in C^\infty(\M)$, $X \in \Gamma(T_\M)$ and $S \in \Gamma(\EE)$. 

The covariant differential associated to a connection $\nabla$ is the map 
\[ \coder : \Omega^\bullet(\M,\EE) \to \Omega^{\bullet+1}(\M,\EE) \] of degree $+1$ satisfying 
\begin{gather*}
\nabla_X S = \iota_X (\coder S) ,\\ 
\coder(\alpha \wedge \beta ) = d \alpha \wedge \beta + (-1)^{|\alpha|} \alpha \wedge \coder \beta ,
\end{gather*}
for all $X\in\Gamma(T_{\M})$, $S\in\Gamma(\EE)$, $\alpha\in\Omega(\M)$ 
and $\beta \in \Omega(\M,\EE)$.

The curvature of a connection $\nabla$ is the $2$-form 
$R^\nabla \in \Omega^2(\M,\End(\EE))$ 
defined by 
\[ R^\nabla(X,Y)=(-1)^{\abs{Y}-1}
\Big\{\nabla_X\nabla_Y- (-1)^{|X||Y|} \nabla_Y\nabla_X -\nabla_{\lie{X}{Y}}\Big\} ,\] 
for all homogeneous $X,Y\in\Gamma(T_\M)$ 
so that, for all $\omega\in\Omega(\M,\EE)$, 
we have $(\coder)^2 \omega= R^\nabla \wedge \omega$.

If $\EE = T_\M$, the torsion of $\nabla$ is the (1,2)-tensor 
$T^\nabla: \sections{T_\M}\times \sections{T_\M} \to \sections{T_\M}$ defined by 
\[ T^\nabla(X,Y)=\nabla_X Y - (-1)^{|X||Y|} \nabla_Y X - [X,Y] ,\] 
for all homogeneous $X,Y\in\Gamma(T_\M)$.
\end{definition}

\section{Formal exponential map} 

The exponential maps defined in terms of geodesics of a connection for ordinary smooth manifolds does not generalize straightforwardly to graded manifolds as the latter only exist 
through their algebras of functions. 
However, it turns out that the the fiberwise 
$\infty$-order jet of the geodesic exponential map 
admits a purely algebraic description, 
which extends readily to the context of graded manifolds~\cite{1408.2903}. 

\begin{definition}\label{Tuileries}
Let $\M$ be a graded manifold and 
let $\mathcal{D}(\M)$ denote its algebra of differential operators. 
The formal exponential map associated to a connection $\nabla$ on $T_\M$
is the morphism of left 
$C^\infty(\M)$-modules 
\[ \pbw:\Gamma(S T_\M) \to \mathcal{D}(\M) ,\] 
inductively defined by the relations 
\begin{gather*}
\pbw(f)=f, \quad \forall f\in C^\infty(\M) \\
\pbw(X)=X,\quad \forall X\in \Gamma(T_\M),
\end{gather*}
and, for all $n\in\NN$ and any homogeneous elements $X_0,\dots,X_n$ of $\sections{T_\M}$, 
\begin{equation}\label{Europe} 
\pbw(X_0 \odot \cdots \odot X_n)= 
\tfrac{1}{n+1}\sum_{k=0}^{n} \epsilon_k\  
\Big\{ X_k \cdot\pbw(X^{\{k\}}) 
-\pbw\big(\nabla_{X_k}(X^{\{k\}})\big) \Big\}
,\end{equation}
where $\epsilon_k=(-1)^{|X_k|(|X_0|+\cdots+|X_{k-1}|)}$ 
and $X^{\{k\}}=X_0\odot\cdots\odot X_{k-1}\odot X_{k+1}\odot\cdots\odot X_n$. 
\end{definition}

\begin{lemma}
The formal exponential map is well-defined.
\end{lemma}

\begin{proof}[Sketch of proof]
Let $\nabla$ be a connection on the tangent bundle $T_\M$ 
to a graded manifold $\M$. 

Consider the sequence $(E_n)_{n\in\NO}$ 
of maps \[ E_n:\underset{n\text{ factors}}{\underbrace{\Gamma(T_\M)\times\Gamma(T_\M)\times\cdots
\times\Gamma(T_\M)}}\to\mathcal{D}(\M) \] 
starting with the natural inclusions 
$E_0:C^\infty(\M)\to\mathcal{D}(\M)$ 
and $E_1:\Gamma(T_\M)\to\mathcal{D}(\M)$ 
and defined recursively by the relation 
\[ \begin{split} 
E_n(X_1,\dots,X_n)=\tfrac{1}{n} \bigg\{ & 
\sum_{k=1}^n \epsilon\ 
X_k\cdot E_{n-1}(X_1,\cdots,\widehat{X_k},\cdots,X_n) \\  & -\sum_{l<k}\epsilon\ E_{n-1}(X_1,\cdots,\nabla_{X_k}X_l, 
\cdots,\widehat{X_k},\cdots,X_n) \\ 
& -\sum_{k<l}\epsilon\ E_{n-1}(X_1,\cdots,\widehat{X_k}, 
\cdots,\nabla_{X_k}X_l,\cdots,X_n) \bigg\} 
,\end{split} \] 
for $n\geqslant 2$.
The symbol $\epsilon$ appearing in each term of the sum 
above denotes the Koszul sign $\epsilon(\sigma;X_1,X_2,\cdots,X_n)$ of the permutation $\sigma$ 
of the order in which the homogeneous elements $X_1,X_2, \dots, X_n$ of $\sections{T_\M}$ appear in that term.

By induction on $n$, show that each $E_n$ is multilinear over $C^\infty(\M)$. 
The morphism of left $C^\infty(\M)$-modules 
\[ E:\Gamma\bigg(\bigoplus_{n\in\NO} (T_\M)^{\otimes n}\bigg)\to\mathcal{D}(\M) \] 
determined by the sequence of maps $(E_n)_{n\in\NO}$ vanishes on the ideal generated by all 
elements of the form 
\[ X_1\otimes\cdots\otimes X_k\otimes X_{k+1}
\otimes\cdots\otimes X_n 
-(-1)^{\abs{X_k}{\abs{X_{k+1}}}} 
X_1\otimes\cdots\otimes X_{k+1}\otimes X_k\otimes\cdots 
\otimes X_n \]
for all homogeneous elements $X_1,\dots,X_n$ of $\Gamma(T_\M)$.
The induced morphism of $C^\infty(\M)$-modules from 
the symmetric algebra $\Gamma\big(S(T_\M)\big)$ to the algebra $\mathcal{D}(\M)$ of differential operators is the formal exponential map: 
\[ \begin{tikzcd}[row sep=tiny]
\Gamma\big(\bigoplus_{n\in\NO} (T_\M)^{\otimes n}\big) 
\arrow{dr}{E} \arrow[two heads]{dd} & \\
& \mathcal{D}(\M) \\
\Gamma\big(S(T_\M)\big) \arrow{ur}[swap]{\pbw} &  
\end{tikzcd} \qedhere \]
\end{proof}

For a classical (i.e.\ nongraded) manifold $M$, this map 
$\pbw$ is the fiberwise $\infty$-order jet of the exponential map 
$\exp:T_M\to M\times M$ associated to the connection $\nabla$ 
--- whence the terminology `formal exponential map.'

\begin{proposition}[\cite{1408.2903}]
If $\M$ is a classical smooth manifold $M$ (i.e.\ $V=0$), then 
\begin{equation*} 
\pbw(X_0 \odot \cdots \odot X_k)(f) \\ 
= \left.\frac{d}{dt_0}\right|_0 \left.\frac{d}{dt_1}\right|_0 \cdots 
\left.\frac{d}{dt_k}\right|_0 
f\big(\exp(t_0 X_0 + t_1 X_1 +\cdots + t_k X_k)\big)
\end{equation*}
for all $X_0, X_1,\cdots, X_k \in\Gamma(T_{M})$ and $f\in C^\infty(M)$. 
\end{proposition}

\section{Properties of the formal exponential map}

The algebra $\mathcal{D}(\M)$ of differential operators 
on $\M$ can be thought of as the universal enveloping algebra $\enveloping{T_\M}$ 
of $T_\M$ (regarded as a Lie algebroid) \cite{MR0154906,MR2653938,MR1687747} 
and admits a natural filtration by the order of the differential operators.

Using \eqref{Europe}, it is straightforward to prove by induction on $n$ that, 
\[ \pbw(X_1\odot\cdots\odot X_n)\in\mathcal{U}^{\leqslant n}(T_\M) ,\] 
for all $n\in\NN$ and $X_1,\dots,X_n\in\sections{T_\M}$. 
In other words, the map $\pbw$ respects the filtrations 
on $\sections{S T_\M}$ and $\enveloping{T_\M}$.

We introduce the functor $\Gr$ which takes a filtered vector space
\[ \cdots\subset\mathscr{A}^{\leqslant k-1}\subset\mathscr{A}^{\leqslant k}\subset\mathscr{A}^{\leqslant k+1}\subset\cdots \] to 
the associated graded vector space 
\[ \Gr\big(\mathscr{A}\big)=\bigoplus_{k} \frac{\mathscr{A}^{\leqslant k}}{\mathscr{A}^{\leqslant k-1}} .\]  

Rinehart proved that, for every Lie algebroid $L$, 
the symmetrization map \[ \sym:\sections{S^\bullet(L)}\to\Gr^\bullet\big(\enveloping{L}\big) ,\] defined by 
\[ X_1\odot \cdots\odot X_n \mapsto \frac{1}{n!}\sum_{\sigma\in S_n} \epsilon\ X_{\sigma(1)}\cdots X_{\sigma(n)} 
,\quad\forall X_1,\dots,X_n\in\sections{L} ,\]
is an isomorphism of graded vector spaces~\cite{MR0154906}. 

The next lemma asserts that \[ \Gr(\pbw)=\sym .\] 

\begin{lemma}\label{Gambetta}
For all $n\in\NN$ and $X_1,\dots,X_n\in\sections{T_\M}$, 
\[ \pbw(X_1\odot\cdots\odot X_n)-\frac{1}{n!}\sum_{\sigma\in S_n} \epsilon\ 
X_{\sigma(1)}\cdot X_{\sigma(2)}\cdot \cdots \cdot X_{\sigma(n)} \]
is an element of $\enveloping{T_\M}^{\leqslant n-1}$.
\end{lemma}

\begin{proof}[Sketch of proof]
It follows from~\eqref{Europe} that 
\begin{multline*} \pbw(X_1\odot\cdots\odot X_n)-\frac{1}{n}\sum_{k=1}^n \epsilon\ X_k\cdot\pbw
(X_1\odot\cdots\odot\widehat{X_k}\odot\cdots\odot X_n) \\ 
=-\frac{1}{n}\sum_{k=1}^n \epsilon\ \pbw\big(\nabla_{X_k}(X_1\odot\cdots
\odot\widehat{X_k}\odot\cdots\odot X_n)\big) \end{multline*}
belongs to $\mathcal{U}^{\leqslant n-1}(T_\M)$ as 
\[ \nabla_{X_k}(X_1\odot\cdots\odot\widehat{X_k}\odot\cdots\odot X_n)
\in\sections{S^{n-1}(T_\M)} \] 
and $\pbw$ respects the filtrations. 
The result follows by induction on $n$. 
\end{proof}

The functor $\Gr$ has the following remarkable property: 
given a homomorphism $\phi:V\to W$ of filtered vector spaces, if the associated morphism of graded vector spaces $\Gr(\phi)$ is an isomorphism and the filtrations on $V$ and $W$ are both exhaustive and complete,\footnote{A filtration $\cdots\subset F^p(V)\subset 
F^{p+1}(V)\subset\cdots$ on a vector space $V$ is said to be exhaustive if 
$\bigcup_{p}F^p(V)=V$ and complete if 
the natural map $V\to\varprojlim_{p}\frac{V}{F^p(V)}$ is an isomorphism.} then $\phi$ itself is an isomorphism. 
Therefore, we have proved the following proposition: 

\begin{proposition}
The formal exponential map 
\[ \pbw:\sections{S T_\M}\to\enveloping{T_\M} \] 
is an isomorphism of filtered left $C^\infty(\M)$-modules. 
\end{proposition}

Note that both $\Gamma(S T_\M)$ and $\enveloping{T_\M}$ are coalgebras over $C^\infty(\M)$. 

The comultiplication 
\[ \Delta:\enveloping{T_\M}\to\enveloping{T_\M}\otimes_{C^\infty(\M)}\enveloping{T_\M} \] 
is characterized by the identities
\begin{gather}
\Delta(1)= 1\otimes 1; \notag \\
\Delta(X)= 1\otimes X + X\otimes 1, \quad\forall X\in\Gamma(T_\M); \notag \\
\Delta(U\cdot V) = \Delta(U)\cdot\Delta(V), \quad\forall U,V\in\enveloping{T_\M}, \label{Crimee}
\end{gather}
and is compatible with the natural filtration of $\enveloping{T_\M}$. 
Here the symbol $\otimes$ denotes the 
tensor product over $C^\infty(\M)$, the symbol 
$1$ denotes the constant function $1$, 
and the symbol $\cdot$ denotes the multiplication in $\enveloping{T_\M}$. 
We refer the reader to~\cite{MR1815717} for the precise 
meaning of Equation~\eqref{Crimee}. 

More explicitly, for all homogeneous elements $X_1,\cdots, X_k\in\Gamma(T_\M)$, 
we have
\begin{multline*}
\Delta(X_1 \cdots X_k)  
= 1 \otimes (X_1 \cdots X_k) + ( X_1 \cdots X_k) \otimes 1 \\
+ \sum_{\substack{p+q=k \\ p,q \in \NN}} \sum_{\sigma \in S_{p,q}} 
\epsilon(\sigma;X) (X_{\sigma(1)} \cdots X_{\sigma(p)}) 
\otimes (X_{\sigma(p+1)} \cdots X_{\sigma(k)}) 
,\end{multline*}
where $\epsilon(\sigma;X)$ denotes the Koszul sign of the permutation $\sigma$ of the homogeneous elements $X_1,\cdots, X_k\in\Gamma(T_\M)$ and $S_{p,q}$ 
denotes the space of $(p,q)$-shuffles.

Similarly, the comultiplication 
\[ \Delta: \Gamma(S T_\M) \to \Gamma(S T_\M) \otimes \Gamma(S T_\M) \]
is given by
\begin{multline*}
\Delta(X_1 \odot \cdots \odot X_k) 
= 1 \otimes (X_1 \odot \cdots \odot X_k) + ( X_1 \odot \cdots \odot X_k) \otimes 1 \\
+ \sum_{\substack{p+q=k \\ p,q \in \NN}} \sum_{\sigma \in S_{p,q}} \epsilon(\sigma;X) (X_{\sigma(1)} \odot \cdots \odot X_{\sigma(p)}) \otimes (X_{\sigma(p+1)} \odot \cdots \odot X_{\sigma(k)}) 
.\end{multline*}

The symbol $\odot$ denotes the symmetric product in $\Gamma(S T_\M)$.

\begin{theorem}\label{Vavin}
The formal exponential map $\pbw:\Gamma(S(T_\M))\to\enveloping{T_\M}$ 
is an isomorphism of filtered coalgebras over $C^\infty(\M)$.
\end{theorem}

\begin{proof}
We need to prove that 
\[ \Delta\circ\pbw=(\pbw\otimes\pbw)\circ\Delta .\]
Given $n\in\NN$ and homogeneous elements $X_0, X_1, \dots, X_n$ of $\sections{T_\M}$, set 
\[ X^{\{k\}}=X_0\odot\cdots\odot X_{k-1}\odot X_{k+1}\odot\cdots\odot X_n \] 
and $\epsilon_k=(-1)^{\abs{X_k}(\abs{X_0}+\cdots+\abs{X_{k-1}})}$ for all $k\in\{0,\dots,n\}$. 

We have 
\begin{align*}
& \Delta\circ\pbw(X_0\odot\cdots\odot X_n) \\ 
=& \frac{1}{n+1} \sum_{k=0}^n \epsilon_k \
\Delta\Big(X_k\cdot\pbw(X^{\{k\}})-\pbw\big(\nabla_{X_k}(X^{\{k\}})\big)\Big) \\ 
=& \frac{1}{n+1} \sum_{k=0}^n \epsilon_k \Big\{
\Delta(X_k)\cdot\Delta\circ\pbw(X^{\{k\}})-\Delta\circ\pbw\big(\nabla_{X_k}(X^{\{k\}})\big) \Big\} \\ 
=& \frac{1}{n+1} \sum_{k=0}^n \epsilon_k \Big\{
(1\otimes X_k+X_k\otimes 1)\cdot(\pbw\otimes\pbw)\circ\Delta(X^{\{k\}}) \\ 
& -(\pbw\otimes\pbw)\circ\Delta\big(\nabla_{X_k}(X^{\{k\}})\big) \Big\} \\ 
=& \frac{1}{n+1} ( \un +\deux -\quatre -\trois ) 
\end{align*}
where 
\begin{align*}
\un &= \sum_{p=-1}^{n-1} \sum_{\substack{\sigma\in S_{n+1} \\ \sigma(0)<\cdots<\sigma(p) \\ 
\sigma(p+2)<\cdots<\sigma(n)}} \epsilon \ \pbw(X_{\sigma(0)}\odot\cdots\odot 
X_{\sigma(p)})\otimes X_{\sigma(p+1)}\cdot\pbw(X_{\sigma(p+2)}\odot\cdots\odot 
X_{\sigma(n)}) 
\\ 
\deux &= \sum_{p=0}^{n} \sum_{\substack{\sigma\in S_{n+1} \\ \sigma(1)<\cdots<\sigma(p) \\ 
\sigma(p+1)<\cdots<\sigma(n)}} \epsilon \ X_{\sigma(0)}\cdot
\pbw(X_{\sigma(1)}\odot\cdots\odot X_{\sigma(p)})
\otimes 
\pbw(X_{\sigma(p+1)}\odot\cdots\odot X_{\sigma(n)}) 
\\ 
\quatre &= \sum_{p=1}^{n} \sum_{\substack{\sigma\in S_{n+1} \\ \sigma(2)<\cdots<\sigma(p) \\ 
\sigma(p+1)<\cdots<\sigma(n)}} \epsilon \ 
\pbw\big( (\nabla_{X_{\sigma(0)}} X_{\sigma(1)})
\odot X_{\sigma(2)}\odot\cdots\odot X_{\sigma(p)}\big) 
\otimes \pbw(X_{\sigma(p+1)}\odot\cdots\odot X_{\sigma(n)}) 
\\ 
\trois &= \sum_{p=-1}^{n-2} \sum_{\substack{\sigma\in S_{n+1} \\ \sigma(0)<\cdots<\sigma(p) \\ 
\sigma(p+3)<\cdots<\sigma(n)}} \epsilon \ \pbw(X_{\sigma(0)}\odot\cdots\odot 
X_{\sigma(p)})\otimes \pbw\big( (\nabla_{X_{\sigma(p+1)}} X_{\sigma(p+2)})
\odot X_{\sigma(p+3)}\odot\cdots\odot X_{\sigma(n)}\big) 
.\end{align*}
Note that, in each term, the factor $\epsilon$ denotes the Koszul sign 
of the permutation $\sigma$ applied to the homogeneous elements 
$X_0, X_1, \dots, X_n$ of $\sections{T_\M}$ in that term.

Setting 
\[ \mathfrak{S}_p = 
\sum_{\substack{\sigma\in S_{n+1} \\ 
\sigma(0)<\cdots<\sigma(p) \\ 
\sigma(p+1)<\cdots<\sigma(n)}} \epsilon \
\pbw(X_{\sigma(0)}\odot\cdots\odot X_{\sigma(p)}) 
\otimes 
\pbw(X_{\sigma(p+1)}\odot\cdots\odot X_{\sigma(n)}) 
,\] 
it follows from \eqref{Europe} that  
\[ \deux-\quatre = \sum_{p=0}^{n} (p+1) \mathfrak{S}_p 
\qquad \text{and} \qquad 
\un-\trois = \sum_{p=-1}^{n-1} (n-p) \mathfrak{S}_p .\] 

Therefore, we obtain 
\begin{multline*} \deux-\quatre +\un-\trois = 
\sum_{p=0}^{n} (p+1) \mathfrak{S}_p
+ \sum_{p=-1}^{n-1} (n-p) \mathfrak{S}_p \\
=\sum_{p=-1}^{n} (p+1) \mathfrak{S}_p
+ \sum_{p=-1}^{n} (n-p) \mathfrak{S}_p 
= (n+1) \sum_{p=-1}^{n} \mathfrak{S}_p 
\end{multline*}
and 
\begin{multline*}
\Delta\circ\pbw(X_0\odot\cdots\odot X_n) 
= \frac{1}{n+1} (\deux-\quatre +\un-\trois) \\
= \sum_{p=-1}^{n} \mathfrak{S}_p 
= (\pbw\otimes\pbw)\circ\Delta
\big(X_0\odot\cdots\odot X_n\big) 
\end{multline*}
since 
\[ \Delta\big(X_0\odot\cdots\odot X_n\big)
= \sum_{p=-1}^{n} 
\sum_{\substack{\sigma\in S_{n+1} \\ 
\sigma(0)<\cdots<\sigma(p) \\ 
\sigma(p+1)<\cdots<\sigma(n)}} \epsilon \
(X_{\sigma(0)}\odot\cdots\odot X_{\sigma(p)}) 
\otimes (X_{\sigma(p+1)}\odot\cdots\odot X_{\sigma(n)}) 
.\qedhere\] 
\end{proof}

\section{Highest order terms of the formal exponential map}

The purpose of this section is to establish the following technical result. 

\begin{proposition}\label{Jasmin}
Let $\M$ be a graded manifold. 
The formal exponential map associated to a \emph{torsion-free} connection $\nabla$ on $T_\M$
satisfies 
\begin{equation}\label{Solferino}
\pbw(X_0\odot\cdots\odot X_n) 
\equiv X_0\cdots X_n - \sum_{j<k} \epsilon\ X_0\cdots\widehat{X_j}\cdots\widehat{X_k}\cdots X_n \cdot 
(\nabla_{X_j}X_k) \mod \mathcal{U}^{\leqslant n-1}(T_\M) 
\end{equation}
and \begin{multline}\label{Stalingrad}
\pbw\inv(X_0\cdots X_n) \equiv X_0\odot\cdots\odot X_n 
\\ \shoveright{+\sum_{j<k} \epsilon\ X_0\odot\cdots\odot\widehat{X_j}\odot\cdots\odot
\widehat{X_k}\odot\cdots\odot X_n\odot 
(\nabla_{X_j}X_k)} \\ \mod \sections{S^{\leqslant n-1}T_\M}
\end{multline}
for all $n\in\NN$ and all homogeneous elements $X_0,\dots,X_n$ in $\sections{T_\M}$.
\end{proposition}

Throughout this section, we make use of the following simplified notations: 
\begin{itemize}
\item The symbol $\epsilon$ appearing in the terms of a sum 
will always denote the Koszul sign $\epsilon(\sigma;X_0,X_1,\cdots,X_n)$ of the permutation $\sigma$ 
of the order in which the homogeneous elements $X_0, X_1, \dots, X_n$ of $\sections{T_\M}$ 
appear in that term.
\item For every subset $\{i_1,i_2,\dots,i_k\}$ of $\{0,1,\dots,n\}$, 
the symbol $X^{\{i_1,i_2,\dots,i_k\}}$ denotes what remains of the product $X_0\cdots X_n$ after 
its factors $X_{i_1},X_{i_2},\dots,X_{i_k}$ have been erased. 
\end{itemize}

\begin{proposition}
For all homogeneous elements $Y,Z$ of $\sections{T_\M}$, we have 
\begin{gather} \label{Corvisart} \pbw(Y\odot Z) = Y\cdot Z -\nabla_Y Z +\tfrac{1}{2} T^\nabla(Y,Z) 
\\ \intertext{and} \label{Glaciere}
\pbw\inv(Y\cdot Z)=Y\odot Z+\nabla_Y Z-\tfrac{1}{2} 
T^\nabla(Y,Z)
.\end{gather}
\end{proposition}

\begin{proof}
It follows directly from Equation~\eqref{Europe} that 
\begin{align*}
\pbw(Y\odot Z) 
=& \tfrac{1}{2} \left\{ Y\pbw(Z)+
(-1)^{\abs{Y}\abs{Z}}
Z\pbw(Y)-\pbw(\nabla_Y Z)-
(-1)^{\abs{Y}\abs{Z}}\pbw(\nabla_Z Y) \right\} \\ 
=& \tfrac{1}{2} \{ Y\cdot Z+ (-1)^{\abs{Y}\abs{Z}} Z\cdot Y-\nabla_Y Z- (-1)^{\abs{Y}\abs{Z}} \nabla_Z Y \} \\ 
=& \tfrac{1}{2} \left\{ 2 YZ-\lie{Y}{Z}-2\nabla_Y Z+(\nabla_Y Z- (-1)^{\abs{Y}\abs{Z}}\nabla_Z Y) \right\} \\ 
=& Y\cdot Z-\nabla_Y Z+\tfrac{1}{2}\left\{\nabla_Y Z-(-1)^{\abs{Y}\abs{Z}}\nabla_Z Y-\lie{Y}{Z}\right\} \\ 
=& Y\cdot Z -\nabla_Y Z +\tfrac{1}{2} T^\nabla(Y,Z) 
.\end{align*}
Therefore, we obtain 
\begin{align*}
Y\cdot Z =& \pbw(Y\odot Z)+\underset{\in\sections{T_\M}}{\underbrace{\nabla_Y Z+\tfrac{1}{2}T^\nabla(Y,Z)}} \\ 
=& \pbw(Y\odot Z) +\pbw\big(\nabla_Y Z+\tfrac{1}{2}
T^\nabla(Y,Z)\big)
\end{align*}
and, applying $\pbw\inv$ to both sides, 
\[ \pbw\inv(Y\cdot Z)= Y\odot Z +\nabla_Y Z+\tfrac{1}{2}
T^\nabla(Y,Z) .\qedhere\]
\end{proof}

\begin{lemma}
For all homogeneous elements $X_1,\dots,X_n$ in $\sections{T_\M}$, we have 
\[ X_1\cdots X_n \equiv \epsilon\ X_1\cdots X_{k-1} X_{k+1} X_k X_{k+2} \cdots X_n 
\mod \mathcal{U}^{\leqslant n-1}(T_\M) .\] 
\end{lemma}

\begin{proof}
It follows from 
$Y\cdot Z=(-1)^{\abs{Y}\abs{Z}}Z\cdot Y+\lie{Y}{Z}$ that 
\begin{equation*} X_1\cdots X_n=\epsilon\ X_1\cdots X_{k-1} X_{k+1} X_k X_{k+2} \cdots X_n 
+ \underset{\in\ \mathcal{U}^{\leqslant n-1}(T_\M)}{\underbrace{X_1\cdots X_{k-1}\lie{X_k}{X_{k+1}} X_{k+2} \cdots X_n}} .\qedhere\end{equation*} 
\end{proof}

\begin{corollary}\label{Pigalle} 
For all homogeneous elements $X_1,\dots,X_n$ in $\sections{T_\M}$, we have 
\begin{gather*} 
X_1\cdots X_n \equiv \epsilon\ X_1 \cdots \widehat{X_k} \cdots X_n \cdot X_k 
\mod \mathcal{U}^{\leqslant n-1}(T_\M) \\ 
\intertext{and, for all $\sigma\in S_n$,} 
X_1\cdots X_n \equiv \epsilon\ X_{\sigma(1)}\cdots X_{\sigma(n)} 
\mod \mathcal{U}^{\leqslant n-1}(T_\M)
.\end{gather*}
\end{corollary}

\begin{corollary}
For all homogeneous elements $X_1,\dots,X_n$ in $\sections{T_\M}$, we have 
\begin{gather} \label{Duroc} \pbw(X_1\odot \cdots \odot X_n)
\equiv X_1\cdots X_n \mod \mathcal{U}^{\leqslant n-1}(T_\M) 
\\ \intertext{and} \label{Dupleix}
X_1\odot \cdots \odot X_n \equiv 
\pbw\inv(X_1\cdots X_n) \mod \sections{S^{\leqslant n-1}T_\M}
.\end{gather} 
\end{corollary}

\begin{proof}
It follows from Lemma~\ref{Gambetta} and Corollary~\ref{Pigalle} that 
\begin{equation*} \pbw(X_1\odot \cdots \odot X_n)
\equiv \frac{1}{n!}\sum_{\sigma\in S_n}\epsilon\ X_{\sigma(1)}\cdots X_{\sigma(n)}  
\equiv X_1\cdots X_n \mod \mathcal{U}^{\leqslant n-1}(T_\M) .\qedhere\end{equation*} 
\end{proof}

\begin{lemma}\label{Simplon}
For all homogeneous elements $X_0,\dots,X_n$ in $\sections{T_\M}$, we have 
\begin{multline*} \sum_{k=0}^n \epsilon\ 
\pbw\big(\nabla_{X_k}(X_0\odot\cdots\odot\widehat{X_k}\odot\cdots\odot X_n)\big) \\
\equiv \sum_{j<k} \epsilon\ X_0\cdots\widehat{X_j}\cdots\widehat{X_k}\cdots X_n \cdot 
\big(2 \nabla_{X_j}X_k -\lie{X_j}{X_k}-T^\nabla(X_j,X_k)\big)
\mod \mathcal{U}^{\leqslant n-1}(T_\M) 
.\end{multline*}
\end{lemma}

\begin{proof}
It follows from Equation~\eqref{Duroc} that  
\begin{multline*} \sum_{k=1}^n \pbw\big(Z_1\odot\cdots\odot Z_{k-1}\odot(\nabla_Y Z_k)\odot Z_{k+1}\odot\cdots\odot Z_n\big) \\ 
\equiv \sum_{k=1}^n Z_1\cdots Z_{k-1}\cdot (\nabla_Y Z_k)\cdot Z_{k+1}\cdots Z_n 
\mod \mathcal{U}^{\leqslant n-1}(T_\M) 
.\end{multline*}
Therefore, for every homogeneous elements $X_0,\dots,X_n$ in $\sections{T_\M}$, we have 
\[ \begin{split} 
\epsilon\ \pbw\big(\nabla_{X_k} & (X_0  \odot\cdots\odot\widehat{X_k}\odot\cdots\odot X_n)\big) \\ 
&\equiv \sum_{j=0}^{k-1}\epsilon\ X_0\cdots X_{j-1}(\nabla_{X_k}X_j)X_{j+1}\cdots\widehat{X_k}\cdots X_n \\
&\quad +\sum_{j=k+1}^{n}\epsilon\ X_0\cdots\widehat{X_k}\cdots X_{j-1}(\nabla_{X_k}X_j)X_{j+1}\cdots X_n 
\mod \mathcal{U}^{\leqslant n-1}(T_\M) .
\end{split} \] 
Summing over $k$, we obtain 
\[ \begin{split}
\sum_{k=0}^n \epsilon\ \pbw\big(\nabla_{X_k} & 
(X_0\odot\cdots\odot\widehat{X_k}\odot\cdots\odot X_n)\big) 
\\ &\equiv \sum_{j<k} \epsilon\ X_0\cdots X_{j-1}(\nabla_{X_k}X_j) X_{j+1}\cdots \widehat{X_k}\cdots X_n \\ 
&\quad + \sum_{k<j} \epsilon\ X_0\cdots \widehat{X_k}\cdots X_{j-1}(\nabla_{X_k}X_j) X_{j+1}\cdots X_n 
\mod \mathcal{U}^{\leqslant n-1}(T_\M) .
\end{split} \]
The desired result follows from Corollary~\ref{Pigalle} and the definition of the torsion $T^\nabla$. 
\end{proof}

\begin{lemma}\label{Pelleport}
For all homogeneous elements $X_0,\dots,X_n$ in $\sections{T_\M}$, we have 
\begin{multline*}
\sum_{k=0}^n \epsilon\ X_k\cdot X_0\cdots \widehat{X_k}\cdots X_n \\ 
\equiv (n+1) X_0\cdots X_n 
+\sum_{j<k}\epsilon\ X_0\cdots \widehat{X_j}\cdots \widehat{X_k}\cdots X_n\cdot 
\lie{X_k}{X_j} \mod \mathcal{U}^{\leqslant n-1}(T_\M) .
\end{multline*}
\end{lemma}

\begin{proof}
We have 
\[ \begin{split} 
\sum_{k=0}^n \epsilon\ X_k\cdot X^{\{k\}} =& 
\sum_{k=0}^n \bigg\{ X_0\cdots X_n + \sum_{j=0}^{k-1} \epsilon\ X_0\cdots X_{j-1}\lie{X_k}{X_j}X_{j+1}\cdots \widehat{X_k}\cdots X_n \bigg\} \\ 
=& (n+1) X_0\cdots X_n + \sum_{j<k} \epsilon\ 
X_0\cdots X_{j-1}\lie{X_k}{X_j}X_{j+1}\cdots\widehat{X_k}\cdots X_n 
.\end{split} \] 
The desired result follows from Corollary~\ref{Pigalle}. 
\end{proof}

\begin{proof}[Proof of Proposition~\ref{Jasmin}]
We will proceed by induction on $n$. 
The result is true for $n=1$ by Equation~\eqref{Corvisart}. 
Now, the induction hypothesis  
\begin{equation*}
\pbw(Z_1\odot\cdots\odot Z_n) 
\equiv Z_1\cdots Z_n - \sum_{j<k} \epsilon\ Z_1\cdots\widehat{Z_j}\cdots\widehat{Z_k}\cdots Z_n \cdot 
(\nabla_{Z_j}Z_k) \mod \mathcal{U}^{\leqslant n-2}(T_\M) 
\end{equation*}
implies   
\begin{multline*} 
\sum_{k=0}^n \epsilon\ X_k\cdot\pbw(X_0\odot\cdot\odot\widehat{X_k}\odot\cdots\odot X_n) \\
\equiv \sum_{k=0}^n \epsilon\ X_k\cdot X_0\cdots\widehat{X_k}\cdots X_n 
- \sum_{k=0}^n \sum_{\substack{i<j \\ k\notin \{i,j\}}} \epsilon\ X_k 
\cdot X^{\{i,j,k\}}\cdot (\nabla_{X_i}X_j) \mod \mathcal{U}^{\leqslant n-1}(T_\M) 
.\end{multline*}
Then, making use of Lemma~\ref{Pelleport} and Lemma~\ref{Pigalle}, we obtain 
\begin{align*} 
& \sum_{k=0}^n \epsilon\ X_k\cdot\pbw(X_0\odot\cdot\odot\widehat{X_k}\odot\cdots\odot X_n) \\
\equiv& (n+1)X_0\cdots X_n+\sum_{j<k}\epsilon\ 
X^{\{j,k\}} 
\lie{X_k}{X_j} -\sum_{k=0}^n \sum_{\substack{i<j \\ k\notin \{i,j\}}} \epsilon\ 
X^{\{i,j\}}  
(\nabla_{X_i}X_j) \\ 
\equiv& (n+1)X_0\cdots X_n-\sum_{j<k}\epsilon\ 
X^{\{j,k\}} 
\lie{X_j}{X_k} -(n-1)\sum_{i<j}\epsilon\ 
X^{\{i,j\}} 
(\nabla_{X_i}X_j) \\
\equiv& (n+1)X_0\cdots X_n-\sum_{j<k}\epsilon\ 
X_0\cdots\widehat{X_j}\cdots\widehat{X_k}\cdots X_n
\big\{\lie{X_j}{X_k} +(n-1) \nabla_{X_j}X_k\big\} \\ 
& \mod \mathcal{U}^{\leqslant n-1}(T_\M) .
\end{align*} 
Combining this last result with the Equation~\eqref{Europe} 
and Lemma~\ref{Simplon}, we finally obtain 
\[ \begin{split} 
\pbw(X_0\odot \cdots X_n) =& \tfrac{1}{n+1} \sum_{k=0}^n \left\{ 
\epsilon\ X_k\cdot\pbw(X^{\{k\}}) - \epsilon\ \pbw\big( \nabla_{X_k}(X^{\{k\}}) \big) \right\} \\ 
\equiv & X_0\cdots X_n - \tfrac{1}{n+1} \sum_{j<k} \epsilon\ 
X^{\{j,k\}} \left\{(n-1) \nabla_{X_j}X_k + 2\nabla_{X_j}X_k -T^\nabla(X_j,X_k)\right\} \\ 
\equiv & X_0\cdots X_n - \sum_{j<k} \epsilon\ 
X_0\cdots\widehat{X_j}\cdots\widehat{X_k}\cdots X_n
\left\{ \nabla_{X_j}X_k - \tfrac{1}{n+1} T^\nabla(X_j,X_k) \right\} \\ 
& \mod \mathcal{U}^{\leqslant n-1}(T_\M) .
\end{split} \] 
The proof of Equation~\eqref{Solferino} is complete since $T^\nabla=0$. 
Finally, applying $\pbw\inv$ to both sides of 
Equation~\eqref{Solferino} and making use of 
Equation~\eqref{Dupleix} yields Equation~\eqref{Stalingrad}.
\end{proof}

\section{Emmrich-Weinstein theorem for graded manifolds}
\label{Washington}

Let $\M$ be a finite-dimensional graded manifold, 
let $(x_i)_{i\in\{1,\dots,n\}}$ be a set of local coordinates on $\M$ 
and let $(y_j)_{j\in\{1,\dots,n\}}$ be the induced local frame of $T_\M\dual$ 
regarded as fiberwise linear functions on 
$T_\M$. As in \cite{MR2102846}, we define
\[ \delta:\Omega^p(\M,S^q T\dual_\M) \to 
\Omega^{p+1}(\M,S^{q-1} T\dual_\M) \] 
and \[ \delta\inv: \Omega^p(\M,S^q T_\M\dual) \to \Omega^{p-1}(\M,S^{q+1} T_\M\dual) \] by 
\[ \delta=\sum_{i=1}^n dx_i\otimes \frac{\partial}{\partial y_i} 
\qquad \text{and} \qquad
\delta\inv=\frac{1}{p+q} \sum_{i=1}^n 
i_{\frac{\partial}{\partial x_i}}\otimes y_i \]
or, more precisely, 
\[ \delta (\omega\otimes f) = \sum_{i=1}^{n} 
(-1)^{\abs{\frac{\partial}{\partial y_i}}\abs{\omega}}\ 
dx_i\wedge\omega \otimes \frac{\partial}{\partial y_i} (f) \]
and 
\[ \delta\inv (\omega\otimes f) = \tfrac{1}{p+q} \sum_{i=1}^{n} (-1)^{\abs{y^i}\abs{\omega}} 
i_{\frac{\partial}{\partial x_{i}}}\omega \otimes y_i\cdot f \]
for all homogeneous $\omega\in\Omega^p(\M)$ 
and for all $f\in\Gamma(S^q T\dual_\M)$.
It's not difficult to check that the operators $\delta$ and $\delta\inv$ 
are well defined, 
i.e.\ independent of the choice of local coordinates, 
and can be extended to $\Omega^\bullet\big(\M,\End(\SM)\big)$. 
The operator $\delta$ has degree $+1$ while the operator $\delta\inv$ has degree $-1$. 
Note that the operators $\delta$ and $\delta\inv$ are \emph{not} inverse of each other. 

A connection $\nabla$ on the tangent bundle $T_\M$ of a graded manifold $\M$ 
determines a connection $\sections{T_\M}\times\sections{S(T_\M)}\to\sections{S(T_\M)}$ 
on $S(T_\M)$, also denoted $\nabla$ by abuse of notation, through the relation 
\[ \nabla_X (X_0\odot X_1 \odot \cdots \odot X_n) = 
\sum_{k=0}^n X_0 \odot \cdots \odot X_{k-1} \odot \nabla_{X}X_k \odot X_{k+1} \odot \cdots \odot X_n .\] 
We use the symbol $\coder$ to denote the covariant differential of the induced connection 
on the dual vector bundle $\hat{S}(T\dual_\M)$. 

Following the construction of Dolgushev \cite{MR2102846}, 
we prove the following proposition.

\begin{proposition}\label{Concorde}
Given a torsion-free connection $\nabla$ on the tangent bundle $T_\M$ of a graded manifold $\M$, 
there exists a unique element  
\[ \Anabla= \sum_{i=1}^n \sum_{\substack{J\in\NO^n \\ |J| \geqslant 2}} \sum_{k=1}^n A_{J,k}^i dx_i \otimes y^J \frac{\partial}{\partial y_{k}} \] 
of degree +1 in $\Omega^1(\M,\hat{S}^{\geqslant 2}(T\dual_\M)\otimes T_\M)$ such that $\delta^{-1}(\Anabla)=0$ 
and the operator \[ \DD:\Omega^\bullet(\M,\SM)\to\Omega^{\bullet+1}(\M,\SM) \] of degree $+1$ 
defined by $\DD=-\delta+\coder+\Anabla$ satisfies $\DD\circ\DD=0$. 
\end{proposition}

Thus we obtain the cochain complex 
\begin{equation}\label{Raspail} \begin{tikzcd}
\Omega^0(\M,\SM) \arrow{r}{\DD} & \Omega^1(\M,\SM) \arrow{r}{\DD} & 
\Omega^2(\M,\SM) \arrow{r}{\DD} & \cdots
\end{tikzcd} \end{equation} 

Note that $\Anabla$ can be thought of as a $1$-form on $\M$ valued in fiberwise formal vector fields on $T_\M$ and hence acts on $\Omega^\bullet(\M,\SM)$, the forms on $\M$ valued in fiberwise formal functions on $T_\M$.

Consider the linear map  
$\sigma: \Omega^\bullet(\M,\SM) \to C^\infty(\M)$ 
of degree $0$ characterized by the relations 
\begin{equation}\label{Passy}
\begin{gathered}
\sigma(f\otimes 1)=f ,\quad\forall f\in C^\infty(\M) ; \\ 
\sigma(\omega \otimes y^J)=0 ,\quad\forall\omega\in\Omega^{\geqslant 1}(\M),\forall J\in\NO^n; \\ 
\sigma(f \otimes y^J)=0 ,\quad\forall f\in\Omega^0(\M),\forall J\in\NO^n 
\text{ s.t. }\abs{J}\geqslant 1
.\end{gathered}
\end{equation}

\begin{proposition}
\label{Voltaire}
For every $f\in C^\infty(\M)$, there exists a unique $\xi\in\Gamma\big(\SM\big)$ 
such that $\sigma(\xi)=f$ and $\DD(\xi)=0$.
\end{proposition}

Hence, there exists a unique map \[ \tau: C^\infty(\M)\to \Omega^0(\M,\SM) \] 
of degree $0$ satisfying $\sigma\circ\tau=\id_{C^\infty(\M)}$ and $\DD\circ\tau=0$. 
Furthermore, due to Proposition~\ref{Voltaire}, one can easily check that $\tau$ 
preserves the algebra structures.
 
\begin{corollary}
The map $\tau: C^\infty(\M)\to \Omega^0(\M,\SM)$ is a morphism of algebras.
\end{corollary} 
 
The cochain complex~\eqref{Raspail} may be augmented to  
\[ \begin{tikzcd}[column sep= small]
0 \arrow{r} & C^\infty(\M) \arrow{r}{\tau} 
& \Omega^0(\M,\SM) \arrow{r}{\DD} 
& \Omega^1(\M,\SM) \arrow{r}{\DD} 
& \Omega^2(\M,\SM) \arrow{r}{\DD} & \cdots
\end{tikzcd} \]

Inspired by the work of Emmrich \& Weinstein~\cite{MR1327535}, 
we proceed to prove that the coboundary operator $\DD$ and 
the map $\tau$ may be obtained directly 
from the formal exponential map $\pbw$. 

We start by defining a connection $\cntn$ on $S T_\M$ by 
\[ \cntn_X S := \pbw\inv\big(X\cdot\pbw(S)\big) \]
for all $X\in\Gamma(T_\M)$ and $S\in\Gamma(S T_\M)$. 

\begin{lemma}\label{Temple}
The connection $\cntn$ is flat.
\end{lemma}

Abusing notations, we will use the same symbol $\cntn$ to denote the induced flat connection on the dual bundle $\SM$.   

\begin{remark}\label{Pasteur}
Equation~\eqref{Glaciere} may be rewritten as 
\[ \cntn_Y Z=Y\odot Z+\nabla_Y Z-\tfrac{1}{2} 
T^\nabla(Y,Z) ,\quad\forall Y,Z\in\sections{T_\M} .\] 
\end{remark}

\begin{theorem}\label{Bastille}
Let $\M$ be a finite-dimensional graded manifold, let $\nabla$ be a \emph{torsionfree} connection on $T_{\M}$ (see Definition~\ref{Philadelphia}), 
and let $\pbw:\sections{ST_{\M}}\to\enveloping{T_\M}$ be the associated formal exponential map (see Definition~\ref{Tuileries}). 
Then the operator $\DD:\Omega^\bullet(\M,\SM)\to\Omega^{\bullet+1}(\M,\SM)$ arising from the torsionfree connection $\nabla$ 
(see Proposition~\ref{Concorde}) is the covariant differential associated with the flat connection $\cntn$ on $\SM$ arising from the 
formal exponential map $\pbw$, i.e.\ $\DD=\ds$. 
\end{theorem}

The proof of Theorem~\ref{Bastille} is deferred to Section~\ref{Convention}.

The operator $\DD$ on $\Omega^\bullet(\M,\SM)$ induces 
similar operators on $\Omega^\bullet(\M,\mathcal{T}_{\poly})$ 
and $\Omega^\bullet(\M,\mathcal{D}_{\poly})$. 
See for instance \cite{MR2102846}. 
Here $\mathcal{T}_{\poly}$ and $\mathcal{D}_{\poly}$
denote the bundles of fiberwise polyvector fields 
and fiberwise polydifferential operators on $T_\M$, respectively.  
These induced operators may be used to `globalize' Kontsevich's (local) formality theorem. 
The details will be discussed in a forthcoming work by Xu and the authors. 

As a corollary of Theorem~\ref{Bastille}, we obtain an extension of a result 
of Emmrich \& Weinstein to graded manifolds (see \cite{MR1327535}). 

\begin{corollary}\label{Bourse}
Let $\M$ be a finite-dimensional graded manifold,
let $(x_i)_{i\in\{1,\dots,n\}}$ be a set of local coordinates on $\M$ and let $(y_j)_{j\in\{1,\dots,n\}}$ be the induced local frame of $T_\M\dual$ regarded as fiberwise linear functions on 
$T_\M$. 
For all $f\in C^\infty(\M)$, we have 
\[ \tau(f) = \sum_{I\in\NO^n} \tfrac{1}{I!} 
y^{I} \otimes \pbw\Big(\underleftarrow{\px^I}\Big)(f) ,\] where 
\[ \underleftarrow{\px^I}=\underset{i_n \text{ factors}}{\underbrace{\partial_{x_n}\odot\cdots\odot\partial_{x_n}}}
\odot\underset{i_{n-1} \text{ factors}}{\underbrace{\partial_{x_{n-1}}\odot\cdots\odot\partial_{x_{n-1}}}}
\odot\cdots\odot\underset{i_1 \text{ factors}}{\underbrace{\partial_{x_1}\odot\cdots\odot\partial_{x_1}}} \] 
for $I=(i_1,i_2,\dots,i_n)\in\NO^n$.
\end{corollary}

\begin{proof}
Straightforward computations yield  
\begin{gather*} 
\sigma\bigg(\sum_{I\in\NO^n} \tfrac{1}{I!} 
y^{I} \otimes \pbw\Big(\underleftarrow{\px^I}\Big)(f)\bigg)=f \\
\intertext{and} 
\ds\bigg(\sum_{I\in\NO^n} \tfrac{1}{I!} 
y^{I} \otimes \pbw\Big(\underleftarrow{\px^I}\Big)(f)\bigg)= 0 
\end{gather*}
for all $f\in C^\infty(\M)$.
Since $\ds=\DD$ and $\tau$ is the only map from $C^\infty(\M)$ to $\Omega^0(\M,\SM)$ 
satisfying $\sigma\circ\tau=\id_{C^\infty(\M)}$ and $\DD\circ\tau=0$, the desired result follows. 
\end{proof}

Specializing to classical (i.e.\ nongraded) manifolds, 
we recover the result of Emmrich \& Weinstein:

\begin{corollary}[Emmrich--Weinstein \cite{MR1327535}]
For a smooth (nongraded) manifold $M$, 
the map \[ \tau: C^\infty(M)\to \Omega^0(M,\SM) \] satisfies 
\[ \tau(f) = \sum_{I \in \NO^n} \tfrac{1}{I!} 
\pbw\big(\px^I\big)(f) \otimes y^{I} 
\]
for all $f\in C^\infty(M)$. 
\end{corollary}

\section{Proof of Theorem~\ref{Bastille}}
\label{Convention}

\begin{lemma}\label{Javel}
Let $\M$ be a graded manifold and let $\nabla$ be 
a torsionfree connection on $T_\M$. 
For all $n\in\NN$ and $X_0,\dots,X_n\in\sections{T_\M}$, we have 
\[ \cntn_{X_0}(X_1\odot\cdots\odot X_n)
\equiv X_0\odot X_1\odot \cdots\odot X_n 
+\nabla_{X_0}(X_1\odot\cdots\odot X_n)
\mod \sections{S^{\leqslant n-1}T_\M} .\]
\end{lemma}

\begin{proof}
According to Proposition~\ref{Jasmin}, we have 
\[ \pbw(X_1\odot\cdots\odot X_n)
\equiv X_1\cdots X_n -\sum_{j<k}\epsilon\ 
X^{\{j,k\}}\cdot(\nabla_{X_j}X_k) \mod 
\mathcal{U}^{\leqslant n-2}(T_\M) .\] 
Multiplying from the left by $X_0$, we obtain 
\begin{align*}
& X_0\cdot \pbw(X_1\odot\cdots\odot X_n) \\ 
\equiv & X_0\cdot X_1\cdots X_n-\sum_{0<j<k}\epsilon\ 
X^{\{j,k\}}\cdot(\nabla_{X_j}X_k) \mod 
\mathcal{U}^{\leqslant n-1}(T_\M) \\ 
\equiv & \bigg\{ X_0\cdots X_n-\sum_{j<k}\epsilon\ 
X^{\{j,k\}}\cdot(\nabla_{X_j}X_k) \bigg\}
+ \sum_{k=1}^n \epsilon\ X^{\{0,k\}}\cdot(\nabla_{X_0}X_k) 
\mod \mathcal{U}^{\leqslant n-1}(T_\M) 
.\end{align*} 
Making use of Equations~\eqref{Solferino} and~\eqref{Duroc}, 
the last congruence above becomes 
\begin{align*}
& X_0\cdot \pbw(X_1\odot\cdots\odot X_n) \\ 
\equiv & \pbw(X_0\odot\cdots\odot X_n)
+ \sum_{k=1}^n \epsilon\ 
\pbw\big(X_1\odot\cdots\odot\widehat{X_k}\odot\cdots
\odot X_n\odot(\nabla_{X_0}X_k)\big) 
\mod \mathcal{U}^{\leqslant n-1}(T_\M) \\ 
\equiv & \pbw\bigg(X_0\odot\cdots\odot X_n 
+ \sum_{k=1}^n \epsilon\ 
X_1\odot\cdots\odot(\nabla_{X_0}X_k)\odot\cdots
\odot X_n \bigg)
\mod \mathcal{U}^{\leqslant n-1}(T_\M) \\ 
\equiv & \pbw\big(X_0\odot\cdots\odot X_n 
+ \nabla_{X_0}(X_1\odot\cdots\odot X_n) \big)
\mod \mathcal{U}^{\leqslant n-1}(T_\M) 
.\end{align*} 
Therefore, we have 
\begin{multline*} \cntn_{X_0}(X_1\odot\cdots\odot X_n)
=\pbw\inv\big(X_0\cdot\pbw(X_1\odot\cdots\odot X_n)\big) \\ \equiv X_0\odot X_1\odot\cdots\odot X_n 
+ \nabla_{X_0}(X_1\odot\cdots\odot X_n) 
\\ \mod \sections{S^{\leqslant n-1}T_\M} 
.\qedhere\end{multline*}
\end{proof}

Lemma~\ref{Javel} above leads us to consider the map 
\[ \nP:\Gamma\big(S(T_\M)\big)\to
\Omega^1\big(\M,S(T_\M)\big) \] 
of degree $+1$ defined by 
\begin{equation}\label{Olympiades}
i_X\nP(S)= \cntn_X S -X\odot S -\nabla_X S 
,\end{equation} 
for all $X\in\sections{T_\M}$ and $S\in\sections{S T_\M}$.

Lemma~\ref{Javel} asserts that, if $\nabla$ is torsionfree,  $i_X\nP$ maps $\Gamma\big(S^k T_\M\big)$ 
to $\Gamma(S^{\leqslant k-1} T_\M)$, 
for all $X\in\sections{T_\M}$ and $k\in\NO$.

\begin{proposition}\label{Wagram}
For every $X\in\sections{T_\M}$, 
the operator $i_X\nP$ is a coderivation of the coalgebra $\sections{S T_\M}$. 
\end{proposition}

\begin{proof}
The operators $\cntn_X$, $\nabla_X$, 
and $S\mapsto X\odot S$ 
are coderivations of $\sections{S T_\M}$. 
\end{proof}

\begin{proposition}\label{Picpus}
If $T^\nabla=0$, then $\nP(S)=0$ 
for all $S\in\sections{S^{\leqslant 1}T_\M}$. 
\end{proposition}

\begin{proof}
It follows from the definitions and Remark~\ref{Pasteur} 
that $\nP(1)=0$ and 
\[ i_Y\nP(Z)=\cntn_Y Z - Y\odot Z -\nabla_Y Z 
= -\tfrac{1}{2} T^\nabla(Y,Z) ,\] 
for all $Y,Z\in\sections{T_\M}$. 
\end{proof}

\begin{proposition}\label{Tolbiac} 
For all $n\in\NN$ and all homogeneous 
$X_0,\dots,X_n\in\sections{T_\M}$, we have 
\[ \sum_{k=0}^n \epsilon\ i_{X_k}\nP(X_0\odot\cdots\odot
\widehat{X_k}\odot\cdots\odot X_n) =0 .\] 
\end{proposition}

\begin{proof} 
Applying $\pbw\inv$ to both sides 
of Equation~\eqref{Europe}, we obtain 
\[ (n+1)\ X_0 \odot\cdots\odot X_n = 
\sum_{k=0}^{n} \epsilon\ 
\Big\{ \pbw\inv\big(X_k\cdot\pbw(X^{\{k\}})\big) 
-\nabla_{X_k}(X^{\{k\}}) \Big\} ,\] 
which we may rewrite as 
\[ \sum_{k=0}^{n} \epsilon\ X_k\odot X^{\{k\}} 
= \sum_{k=0}^{n} \epsilon\ 
\Big\{ \cntn_{X_k} (X^{\{k\}}) 
-\nabla_{X_k}(X^{\{k\}}) \Big\} .\]
The desired result follows from Equation~\eqref{Olympiades}.
\end{proof}

Now consider the map 
\[ \nA : \sections{S^\bullet (T\dual_\M)} \to 
\Omega^1\big(\M,\hat{S}^{\geqslant \bullet+1} 
(T\dual_\M)\big) \]
of degree $+1$ defined by 
\begin{equation}\label{Rambuteau} 
\contraction{S}{i_X \nA(\sigma)}
=(-1)^{\abs{S}\abs{X}}\contraction{i_X\nP(S)}{\sigma} 
,\end{equation}
for all homogeneous $X\in\sections{T_\M}$, $S\in\sections{S T_\M}$, and $\sigma\in\Gamma(\SM)$. 
Here 
\[ \Gamma(S T_\M)\otimes_{C^\infty(\M)}
\Gamma(\SM)\xto{\contraction{-}{-}} 
C^\infty(\M) \] 
is the duality pairing defined by 
\[ \contraction{X_1\odot\cdots\odot X_p}
{\alpha_1\odot\cdots\odot\alpha_q} = 
\begin{cases} 
\sum_{\sigma\in S_p} \epsilon\ 
i_{X_1}\alpha_{\sigma(1)}\cdot
i_{X_2}\alpha_{\sigma(2)}\cdots 
i_{X_p}\alpha_{\sigma(p)} 
& \text{if } p=q \\ 
0 & \text{if } p\neq q 
\end{cases} \] 
for all homogeneous $X_1,\dots,X_p\in\sections{T_\M}$ 
and $\alpha_1,\dots,\alpha_q\in\sections{T_\M}$. 
The factor $\epsilon$ in the equation above
denotes the Koszul sign of the 
permutation of homogeneous elements 
\[ X_1,X_2,\dots,X_p,\alpha_1,\alpha_2,\dots,\alpha_p 
\quad\longmapsto\quad X_1,\alpha_{\sigma(1)},X_2,\alpha_{\sigma(2)},
\dots,X_p,\alpha_{\sigma(p)} .\] 

A straightforward computation yields the following lemma. 

\begin{lemma}\label{Bercy}
Let $\M$ be a finite-dimensional graded manifold,
let $(x_i)_{i\in\{1,\dots,n\}}$ be a set of local coordinates on $\M$ and let $(y_j)_{j\in\{1,\dots,n\}}$ be the induced local frame of $T_\M\dual$ regarded as fiberwise linear functions on 
$T_\M$. For all $I,J\in\NO^n$ such that $\abs{I}=\abs{J}$, we have 
\[ \contraction{\underleftarrow{\partial_x^I}}{y^J}=I!\ \delta_{I,J} .\]  
\end{lemma}

\begin{lemma}\label{Goncourt}
For all homogeneous $X\in\sections{T_\M}$, $S\in\sections{S T_\M}$, and $\sigma\in\Gamma\big(\hat{S}(T\dual_\M)\big)$, we have
\[ \contraction{S}{i_X\delta(\sigma)}
=(-1)^{\abs{S}\abs{X}}\contraction{X\odot S}{\sigma}
.\] 
\end{lemma}

\begin{proof}
It suffices to prove the relation for 
$S=\underleftarrow{\partial_x^I}$, $\sigma=y^J$, 
and $X=\frac{\partial}{\partial x_l}$. 
We have 
\begin{align*}
\contraction{\underleftarrow{\partial_x^I}}
{i_{\frac{\partial}{\partial x_l}}\delta(y^J)} 
=& 
\contraction{\underleftarrow{\partial_x^I}}
{\sum_{k=1}^n i_{\frac{\partial}{\partial x_l}} 
dx_k\ \frac{\partial}{\partial y_k}(y^J)} \\ 
=& \contraction{\underleftarrow{\partial_x^I}}
{\frac{\partial}{\partial y_l}(y^J)} \\ 
=& \contraction{\underleftarrow{\partial_x^I}}
{(-1)^{\abs{y_l}\abs{y^{\truncation{<l}J}}} 
\ j_l\ y^{J-e^l}} \\ 
=& (-1)^{\abs{y_l}\abs{y^{\truncation{<l}J}}} 
\ j_l\ I!\ \delta_{I,J-e^l} \\
=& (-1)^{\abs{y_l}\abs{y^{\truncation{<l}J}}}\ J!\ 
\delta_{I+e_l,J} \\ 
=& (-1)^{\abs{\frac{\partial}{\partial x_l}}
\abs{\partial_x^{\truncation{<l}I}}}
\contraction{\underleftarrow{\partial_x^{I+e_l}}}{y^J} \\ 
=& \contraction{\underleftarrow{\partial_x^I}\odot
\frac{\partial}{\partial x_l}}{y^J} \\ 
=& (-1)^{\abs{\underleftarrow{\partial_x^I}}
\abs{\frac{\partial}{\partial x_l}}} 
\contraction{\frac{\partial}{\partial x_l}\odot \underleftarrow{\partial_x^I}}{y^J}
.\qedhere\end{align*}
\end{proof}

\begin{proposition} 
The operator $i_X\nA$ is a derivation of the algebra $\Gamma\big(\hat{S}(T\dual_\M)\big)$ for every 
$X\in\sections{T_\M}$. 
\end{proposition}

\begin{proof}
The result follows immediately from Proposition~\ref{Wagram} since the algebra $\Gamma\big(\hat{S}(T\dual_\M)\big)$ is dual to the coalgebra 
$\Gamma(S T_\M)$ and $i_X\nA$ is the transpose of 
$i_X\nP$ according to Equation~\eqref{Rambuteau}. 
\end{proof}

Hence $\nA$ may be regarded as an element of 
$\Omega^1\big(\M,\hat{S}(T\dual_\M)\otimes T_\M\big)$. 

\begin{proposition}
If $T^\nabla=0$, then 
$\nA\in\Omega^1\big(\M,\hat{S}^{\geqslant 2}(T\dual_\M)\otimes T_\M\big)$.
\end{proposition}

\begin{proof} 
Let $(x_i)_{i\in\{1,\dots,n\}}$ be a set of local coordinates on $\M$ and let $(y_j)_{j\in\{1,\dots,n\}}$ be the induced local frame of $T_\M\dual$ regarded as fiberwise linear functions on $T_\M$.
Since $i_X\nA$ is a derivation of the algebra
$\Gamma\big(\hat{S}(T\dual_\M)\big)$, 
which is generated by $y_1,\dots,y_n$, 
we have \[ i_X\nA=\sum_{k=1}^n i_X\nA(y_k) \frac{\partial}{\partial y_k} ,\] 
with 
\begin{align*} 
i_X\nA(y_k) =& \sum_{I\in\NO^n} \frac{1}{I!} 
(1\otimes y^I) \cdot (\contraction{\underleftarrow{\partial_x^I}}{i_X\nA(y_k)} \otimes 1)
&&\text{by Lemma~\ref{Bercy},} \\ 
=& \sum_{I\in\NO^n}  \frac{(-1)^{\abs{\underleftarrow{\partial_x^I}}\abs{X}}}{I!} 
(1\otimes y^I) \cdot 
(\contraction{i_X\nP\Big(\underleftarrow{\partial_x^I}\Big)}{y_k}  
\otimes 1)
&&\text{by Equation~\eqref{Rambuteau}.} 
\end{align*}
Since $T^\nabla=0$, it follows from Proposition~\ref{Picpus} that $i_X\nA\in\Gamma\big(\hat{S}^{\geqslant 2}(T\dual_\M)\otimes T_\M\big)$ as $\nP\Big(\underleftarrow{\partial_x^I}\Big)=0$ 
for $\abs{I}\leqslant 1$. 
\end{proof}

\begin{proposition}\label{Telegraphe}
$\delta\inv(\nA)=0$
\end{proposition}

\begin{proof}
Let $(x_i)_{i\in\{1,\dots,n\}}$ be a set of local coordinates on $\M$ and let $(y_j)_{j\in\{1,\dots,n\}}$ be the induced local frame of $T_\M\dual$ regarded as fiberwise linear functions on $T_\M$.

From \[ \nA=\sum_{k=1}^n \sum_{J\in\NO^n} \frac{1}{J!} \ (1\otimes y^J) 
\cdot ( \contraction{\underleftarrow{\partial_x^J}}{\nA(y_k)} \otimes \frac{\partial}{\partial y_k})  ,\]  
we obtain 
\begin{align*}
\delta\inv(\nA) =& 
\sum_{k=1}^n \sum_{J\in\NO^n} \sum_{l=1}^n \frac{1}{J!} 
\ y_l y^J \contraction{\underleftarrow{\partial_x^J}}
{i_{\frac{\partial}{\partial x_l}}\nA(y_k)} 
\frac{\partial}{\partial y_k} 
\\ =& 
\sum_{k=1}^n \sum_{J\in\NO^n} \sum_{l=1}^n \frac{1}{J!} 
(-1)^{\abs{y_l}\abs{y^{\truncation{\leqslant l}J}}} y^{J+e_l}
(-1)^{\abs{\frac{\partial}{\partial x_l}}
\abs{\partial_x^J}} 
\contraction{i_{\frac{\partial}{\partial x_l}}\nP
\Big(\underleftarrow{\partial_x^J}\Big)}{y_k} 
\frac{\partial}{\partial y_k} \\ 
=& \sum_{k=1}^n \sum_{J\in\NO^n} \sum_{l=1}^n \frac{1}{J!} 
\ y^{J+e_l}
(-1)^{\abs{\frac{\partial}{\partial x_l}}
\abs{\partial_x^{\truncation{>l}J}}} 
\contraction{i_{\frac{\partial}{\partial x_l}}\nP
\Big(\underleftarrow{\partial_x^J}\Big)}{y_k} 
\frac{\partial}{\partial y_k} \\ 
=& \sum_{k=1}^n \sum_{M\in\NO^n} \frac{1}{M!} 
\ y^M \contraction{ \sum_{l=1}^n m_l 
(-1)^{\abs{\frac{\partial}{\partial x_l}}
\abs{\partial_x^{\truncation{>l}M}}} 
i_{\frac{\partial}{\partial x_l}}\nP
\Big(\underleftarrow{\partial_x^{M-e_l}}\Big)}{y_k} 
\frac{\partial}{\partial y_k} 
.\end{align*}
It follows directly from Proposition~\ref{Tolbiac} that  
\[ \sum_{l=1}^n m_l \ (-1)^{\abs{\frac{\partial}{\partial x_l}}
\abs{\partial_x^{\truncation{>l}M}}} \ 
i_{\frac{\partial}{\partial x_l}}\nP\Big(\underleftarrow{\partial_x^{M-e_l}}\Big) =0 \]
for every $M=(m_1,\dots,m_n)\in\NO^n$. 
\end{proof}

\begin{proof}[Proof of Theorem~\ref{Bastille}] 
The connections $\nabla$ and $\cntn$ defined on $S(T_\M)$ 
induce connections on the dual bundle $\hat{S}(T\dual_\M)$: 
\[ \contraction{\cntn_X S}{\sigma}+
(-1)^{\abs{X}\abs{S}}
\contraction{S}{\cntn_X \sigma} = 
X\big(\contraction{S}{\sigma}\big)=
\contraction{\nabla_X S}{\sigma}+
(-1)^{\abs{X}\abs{S}}
\contraction{S}{\nabla_X \sigma} .\]
Therefore, we obtain 
\begin{align*}
\contraction{\cntn_X S-\nabla_X S}{\sigma} &=
(-1)^{\abs{X}\abs{S}}
\contraction{S}{\nabla_X \sigma-\cntn_X \sigma} \\ 
\contraction{X\odot S+i_X\nP(S)}{\sigma} &= 
(-1)^{\abs{X}\abs{S}}
\contraction{S}{i_X\big(\coder\sigma-\ds\sigma\big)} \\ 
\intertext{and, making use of Lemma~\ref{Goncourt} 
and Equation~\eqref{Rambuteau},} 
\contraction{S}{i_X\big(\delta\sigma
+\nA\sigma\big)} &= 
\contraction{S}{i_X\big(\coder\sigma-\ds\sigma\big)} 
\end{align*}
or, equivalently, \[ \ds=-\delta+\coder-\nA .\] 
Since $\delta\inv(\nA)=0$ (Proposition~\ref{Telegraphe}) 
and $\ds\circ\ds=0$ (Proposition~\ref{Temple}), 
Theorem~\ref{Concorde} asserts that  
$\Anabla=-\nA$ and $\DD=\ds$. 
\end{proof}

\begin{corollary}
$\Anabla=-\nA$
\end{corollary} 

\section{Dolgushev--Fedosov resolution via homological perturbation}

The following theorem is an analogue of a theorem of \cite{MR2102846} 
transposed to graded manifolds. 
Dolgushev's proof relies on Fedosov's iteration technique. 
Here we give a proof based on homological perturbation (see Appendix).

\begin{theorem}\label{Trocadero}
Suppose $\M$ is a graded manifold of finite dimension and $\nabla$ is 
a torsion-free connection on $T_\M$. 
Let $\Anabla$, $\DD$, and $\tau$ be the operators arising from $\nabla$ 
as explained in Section~\ref{Washington}. 
Then the cochain complex 
\[ \begin{tikzcd}
\Omega^0(\M,\SM) \arrow{r}{\DD} 
& \Omega^1(\M,\SM) \arrow{r}{\DD} 
& \Omega^2(\M,\SM) \arrow{r}{\DD} 
& \cdots 
\end{tikzcd} \] 
together with the augmentation map $\tau:C^\infty(\M)\to\Omega^0(\M,\SM)$ 
is a resolution of $C^\infty(\M)$. 
Moreover, we have 
\[ \tau = \sum_{n=0}^\infty \big(\delta\inv\circ(\coder+\Anabla)\big)^n\circ i ,\]
where $i$ denotes the canonical inclusion of $C^\infty(\M)$ into $\Omega^0(\M,\SM)$ 
and $\delta\inv$ is the operator defined at the beginning of Section~\ref{Washington}.  
\end{theorem}

\begin{remark}
More precisely, we have a contraction (see Appendix)
\[ \begin{tikzcd} 
\big(\Omega^\bullet(\M,\SM), D\big)
\arrow[rr, shift left, "\sigma"] \arrow[loop left, "h"] 
&& \big(C^\infty(\M),0\big) \arrow[ll, shift left, "\tau"] 
\end{tikzcd} \]
with homotopy operator $h:\Omega^\bullet(\M,\SM)\to\Omega^{\bullet-1}(\M,\SM)$ 
defined by \[ h=\sum_{n=0}^\infty \big(\delta\inv\circ(\coder+\Anabla)\big)^n\circ\delta\inv .\qedhere\]
\end{remark}

\begin{proof}
It is not difficult to check that the cochain complex 
$\big(\Omega^\bullet(\M,\SM),-\delta\big)$ 
deformation retracts onto $C^\infty(\M)$: 
the canonical inclusion 
$i: C^\infty(\M)\to\Omega^0(\M,\SM)$ 
and the linear map  
$\sigma: \Omega^\bullet(\M,\SM) \to C^\infty(\M)$ 
characterized by Equations~\eqref{Passy}
satisfy
\[ \sigma i=\id_{C^\infty(\M)}
\qquad\text{and}\qquad 
\delta\delta\inv+\delta\inv\delta=
\id_{\Omega^\bullet(\M,\SM)}-i\sigma .\]
Furthermore, the maps $\sigma$, $i$, and $\delta\inv$ respect the exhaustive, complete,
descending filtrations on $C^\infty(\M)$ and the complex 
$\big(\Omega^\bullet(\M,\SM),-\delta\big)$ respectively 
defined by \[ \mathcal{F}^m=\begin{cases} 
C^\infty(\M) & \text{if } m\leqslant 0 \\ 
0 & \text{if } m>0 \end{cases} \] 
and \[ \mathscr{F}^m= \prod_{p+q\geqslant m} \Omega^p\big(\M,S^q(T\dual_{\M})\big) .\]

More precisely, the diagram 
\[ \begin{tikzcd} 
0 \arrow{r} & \Omega^0(\M,\SM) \arrow{d}{\sigma} \arrow{r}{-\delta} & 
\Omega^1(\M,\SM)  \arrow{d}{\sigma} \arrow{r}{-\delta} 
\arrow[dashed]{ddl}[near end]{\delta\inv} & 
\Omega^2(\M,\SM) \arrow{d}{\sigma} \arrow{r} 
\arrow[dashed]{ddl}[near end]{\delta\inv} & \cdots \\ 
0 \arrow{r} & C^\infty(\M) \arrow{d}[swap]{i} \arrow{r} & 
0 \arrow{d}[swap]{i} \arrow{r} & 
0 \arrow{d}[swap]{i} \arrow{r} & \cdots \\ 
0 \arrow{r} & \Omega^0(\M,\SM) \arrow{r}{-\delta} & 
\Omega^1(\M,\SM)  \arrow{r}{-\delta} & 
\Omega^2(\M,\SM) \arrow{r} & \cdots 
\end{tikzcd} \] 
is a filtered contraction. 

The operator $\perturbation:=\coder+\Anabla$ is a perturbation   
of the differential $-\delta$ on $\Omega^\bullet(\M,\SM)$, 
for $\DD=-\delta+\perturbation$ satisfies $\DD\circ\DD=0$ and 
$\perturbation(\mathscr{F}^m)\subset\mathscr{F}^{m+1}$ for all $m\in\NO$. 

Hence homological perturbation (see Appendix) yields the contraction 
\[ \begin{tikzcd} 
0 \arrow{r} & \Omega^0(\M,\SM) \arrow{d}{\breve{\sigma}} \arrow{r}{\DD} & 
\Omega^1(\M,\SM) \arrow{d}{\breve{\sigma}} \arrow{r}{\DD} 
\arrow[dashed]{ddl}[near end]{h} & 
\Omega^2(\M,\SM) \arrow{d}{\breve{\sigma}} \arrow{r} 
\arrow[dashed]{ddl}[near end]{h} & \cdots \\ 
0 \arrow{r} & C^\infty(\M) \arrow{d}[swap]{\breve{\imath}} \arrow{r} & 
0 \arrow{d}[swap]{\breve{\imath}} \arrow{r} & 
0 \arrow{d}[swap]{\breve{\imath}} \arrow{r} & \cdots \\ 
0 \arrow{r} & \Omega^0(\M,\SM) \arrow{r}{\DD} & 
\Omega^1(\M,\SM)  \arrow{r}{\DD} & 
\Omega^2(\M,\SM) \arrow{r} & \cdots 
\end{tikzcd} \] 
where 
\begin{gather*} 
\breve{\sigma}=\sum_{k=0}^\infty\sigma(\partial\delta\inv)^k \qquad \qquad
\breve{\imath}=\sum_{k=0}^{\infty}(\delta\inv\partial)^k i \\ 
\intertext{and} h=\sum_{k=0}^{\infty}(\delta\inv\partial)^k \delta\inv 
.\end{gather*} 
In particular, we have $\DD\breve{\imath}=0$ and, since $\sigma\delta\inv=0$,
\[\sigma\breve{\imath}=\sigma \sum_{n=0}^{\infty}(\delta\inv\partial)^n i =\sum_{n=0}^{\infty}\sigma(\delta\inv\partial)^n i =\sigma i = \id_{C^\infty(\M)} .\] 
It follows from Proposition~\ref{Voltaire} that \[ \tau = \breve{\imath} = 
\sum_{n=0}^{\infty}(\delta\inv\partial)^n i 
= \sum_{n=0}^{\infty}(\delta\inv(\coder+A))^n i .\]
We note that $\breve{\sigma}=\sigma$ since $\sigma\partial\delta\inv=0$.  
\end{proof}

Recall that a dg-manifold is a graded manifold $\M$ endowed with a vector field $Q$ of degree +1 such that $[Q,Q]=0$. Hence the algebra of functions $C^\infty(\M)$ on a dg-manifold $(\M,Q)$ is a cochain complex with the vector field $Q$ as coboundary operator. 
Two dg-manifolds $(\M_1,Q_1)$ and $(\M_2,Q_2)$ are said to be weakly equivalent if their associated cochain complexes $(C^\infty(\M_1),Q_1)$ and $(C^\infty(\M_2),Q_2)$ are quasi-isomorphic. 
In this terminology, Theorem~\ref{Trocadero} can be rephrased as follows: 
the chain map $\tau$ is a weak equivalence 
of dg-manifolds from $(\M,0)$ to $(T_\M[1] \oplus T_\M,D)$.

\section*{Appendix: homological perturbation}

Roughly speaking, homological perturbation is an algebraic tool which allows us to perturb a deformation retract to another deformation retract. To be precise, we need some technical definitions from homological algebra. We say that a cochain complex $(N,\delta)$ \emph{contracts} onto a cochain complex 
$(M,d)$ if there exists two chain maps $\sigma:N\to M$ and $\tau:M\to N$ 
and an endomorphism $h:N\to N[-1]$ of the graded module $N$ 
satisfying 
\[ \sigma\tau=\id_N, \qquad \tau\sigma-\id_M=h\delta+\delta h \] 
and 
\[ \sigma h=0, \qquad h\tau=0, \qquad h h=0 .\]
If, furthermore, the cochain complexes $N$ and $M$ are filtered and the maps 
$\sigma$, $\tau$, and $h$ preserve the filtration, the contraction is said to be 
filtered~\cite[Section~12]{MR0056295}. 

A filtration $\cdots\subset F_{p-1}N\subset F_pN\subset F_{p+1}N\subset\cdots$ on a cochain complex $N$ is said to be 
exhaustive if $N=\bigcup_p F_p N$ and complete if $N=\varprojlim\frac{N}{F_p N}$. 

A \emph{perturbation} of the differential $\delta$ of a filtered cochain complex 
\[ \begin{tikzcd} \cdots \arrow{r} & N^{n-1} 
\arrow{r}{\delta} & N^{n} \arrow{r}{\delta} & N^{n+1} \arrow{r} & \cdots \end{tikzcd} \] 
is an operator $\partial:F_{\bullet} N \to F_{\bullet-1} N$ lowering the filtration and satisfying \[ (\delta+\partial)(\delta+\partial)=0 \] 
so that $\delta+\partial$ is a new differential on $N$. 

We refer the reader to~\cite[Section~1]{MR1109665} for a brief history of the following proposition. 

\begin{proposition}[Homological Perturbation~\cite{MR0220273}]
Let 
\[ \begin{tikzcd}[column sep=large] 
\cdots \arrow{r} & N^{n-1} \arrow{d}{\sigma} 
\arrow{r}{\delta} & 
N^{n} \arrow{d}{\sigma} \arrow{r}{\delta} 
\arrow[dashed]{ddl}[near end]{h} & 
N^{n+1} \arrow{d}{\sigma} \arrow{r} 
\arrow[dashed]{ddl}[near end]{h} & \cdots \\ 
\cdots \arrow{r} & M^{n-1} \arrow{d}[swap]{\tau} 
\arrow{r}[near start]{d} & 
M^{n} \arrow{d}[swap]{\tau} \arrow{r}[near start]{d} & 
M^{n+1} \arrow{d}[swap]{\tau} \arrow{r} & \cdots \\ 
\cdots \arrow{r} & N^{n-1} \arrow{r}{\delta} & 
N^{n} \arrow{r}{\delta} & 
N^{n+1} \arrow{r} & \cdots 
\end{tikzcd} \] 
be a filtered contraction. 
Given a perturbation $\partial$ of the differential $\delta$ on $N$, 
if the filtrations on $M$ and $N$ are exhaustive and complete, 
then the series 
\begin{gather*}
\vartheta:=\sum_{k=0}^{\infty} \sigma \partial (h\partial)^k\tau = \sum_{k=0}^{\infty}\sigma(\partial h)^k\partial\tau \\ 
\breve{\sigma}:=\sum_{k=0}^{\infty}\sigma(\partial h)^k \\ 
\breve{\tau}:=\sum_{k=0}^{\infty}(h\partial)^k\tau \\ 
\breve{h}:=\sum_{k=0}^{\infty}(h\partial)^k h=\sum_{k=0}^{\infty} h(\partial h)^k
\end{gather*}
converge, $\vartheta$ is a perturbation of the differential $d$ on $M$, 
and 
\[ \begin{tikzcd}[column sep=large] 
\cdots \arrow{r} & N^{n-1} \arrow{d}{\breve{\sigma}} \arrow{r}{\delta+\partial} & 
N^{n} \arrow{d}{\breve{\sigma}} \arrow{r}{\delta+\partial} \arrow[dashed]{ddl}[near end]{\breve{h}} & 
N^{n+1} \arrow{d}{\breve{\sigma}} \arrow{r} 
\arrow[dashed]{ddl}[near end]{\breve{h}} & \cdots \\ 
\cdots \arrow{r} & M^{n-1} \arrow{d}[swap]{\breve{\tau}} \arrow{r}[near start]{d+\vartheta} & 
M^{n} \arrow{d}[swap]{\breve{\tau}} 
\arrow{r}[near start]{d+\vartheta}  & 
M^{n+1} \arrow{d}[swap]{\breve{\tau}} \arrow{r} & \cdots \\ 
\cdots \arrow{r} & N^{n-1} \arrow{r}{\delta+\partial} & 
N^{n} \arrow{r}{\delta+\partial} & 
N^{n+1} \arrow{r} & \cdots 
\end{tikzcd} \] 
constitutes a new filtered contraction. 
\end{proposition}

\bibliographystyle{amsplain}
\bibliography{references}
%\printbibliography

\end{document}